%% file: unipot.tex
\documentclass{amsart}
\usepackage[utf8]{inputenc}
\usepackage{a4}
\usepackage{graphics}
\usepackage{color}
\usepackage{amssymb}
\usepackage[all]{xy}
\usepackage{amsmath}
\usepackage{stmaryrd}
\usepackage{subfigure}
\usepackage{longtable}
\usepackage{setspace}
\usepackage{booktabs}

\usepackage{pst-plot}
\usepackage{pstricks}
\usepackage{color}

\CompileMatrices

\newtheorem{theorem}{Theorem}[section]
\newtheorem{lemma}[theorem]{Lemma}
\newtheorem{proposition}[theorem]{Proposition}
\newtheorem{corollary}[theorem]{Corollary}
\theoremstyle{definition}

\newtheorem{example}[theorem]{Example}

\newtheorem{construction}[theorem]{Construction}

\newtheorem{remark}[theorem]{Remark}
\theoremstyle{remark}

\numberwithin{equation}{section}
\def\Chi{{\mathbb X}}

\def\div{{\rm div}}

\def\quot{/\!\!/}
\def\mal{\! \cdot \!}

\def\rq#1{\widehat{#1}}

\def\b#1{\overline{#1}}

\def\CC{{\mathbb C}}
\def\KK{{\mathbb K}}

\def\ZZ{{\mathbb Z}}

\def\QQ{{\mathbb Q}}
\def\PP{{\mathbb P}}
\def\Of{{\mathcal{O}}}

\def\WDiv{\operatorname{WDiv}}

\def\Cl{\operatorname{Cl}}

\def\Spec{{\rm Spec}}

\def\SL{{\rm SL}}

\newcounter{itemnumber}

\begin{document}
\title[Categorical quotients]{Factorial algebraic group actions\\
and categorical quotients}
\author[I.~V.~Arzhantsev]{Ivan V. Arzhantsev} \address{Department of
  Higher Algebra, Faculty of Mechanics and Mathematics, Moscow State
  University, Leninskie Gory 1, Moscow, 119991, Russia}
\email{arjantse@mccme.ru}
\author[D.~Celik]{Devrim Celik} 
\address{Mathematisches Institut, Universit\"at T\"ubingen,
Auf der Morgenstelle 10, 72076 T\"ubingen, Germany}
\email{celik@mail.mathematik.uni-tuebingen.de}
\author[J.~Hausen]{J\"urgen Hausen} 
\address{Mathematisches Institut, Universit\"at T\"ubingen,
Auf der Morgenstelle 10, 72076 T\"ubingen, Germany}
\email{hausen@mail.mathematik.uni-tuebingen.de}

\begin{abstract}
Given an action of an affine algebraic 
group with only trivial characters 
on a factorial variety,
we ask for categorical quotients.
We characterize existence in the category 
of algebraic varieties.
Moreover, allowing constructible sets as quotients, 
we obtain a more general existence 
result, which, for example, settles the case of a 
finitely generated algebra of invariants.
As an application, we provide a 
combinatorial GIT-type construction of 
categorial quotients for actions on, e.g. complete
varieties with finitely generated Cox ring via lifting 
to the characteristic space.
\end{abstract}

\maketitle

\section{Introduction}

Consider the action of an affine algebraic 
group $G$ on a normal variety $X$ defined 
over an algebraically closed field $\KK$.
In most cases, the orbit space $X/G$ does not 
inherit the structure of a variety and it 
is the main task of Geometric Invariant Theory
to provide reasonable replacements.
A common concept with minimal requirements 
is the {\em categorical quotient}: this is a 
$G$-invariant morphism $\pi \colon X \to Y$ 
such that for every other $G$-invariant 
morphism $\varphi \colon X \to Z$,  
there exists a unique morphism $\psi \colon Y \to Z$ 
with $\varphi = \psi \circ \pi$.
If $G$ is reductive and $X$ is affine, then
Hilbert's finiteness theorem guarantees existence
of a categorical quotient $\pi \colon X \to Y$
with $Y := \Spec \, \Gamma(X, \mathcal{O})^G$
which, in general, is not an orbit space.
However, as soon as one of the conditions 
``$G$ reductive'' and ``$X$ affine'' is not 
satisfied, even categorical quotients need not 
exist any more, see Example~\ref{nocatquot} for 
the first and~\cite{ACHa} for the second one.
In this article, we investigate existence of 
categorical quotients in the following setting:
we say that the action of $G$ on $X$ is 
{\em factorial\/}, if every invariant 
hypersurface $D \subseteq X$ is the zero 
set of an invariant function 
$f \in \Gamma(X,\mathcal{O})$; 
compare~\cite{ReRi} for a related concept.
For example, if $G$ has trivial character 
group $\Chi(G)$, e.g., is semisimple or 
unipotent, and $X$ has finite divisor class group, 
e.g., is a vector space, then the $G$-action
is factorial.

Similarly to the reductive case, the
algebra of invariants plays a central role
in the construction of quotients, 
compare also the work on unipotent group 
actions~\cite{Fa}, \cite{Wi1}, \cite{Wi2}
and~\cite{DoKi}.
In contrast to the reductive case, even 
for affine $X$, the algebra of invariants 
need not be finitely generated. 
However, in our setting there always exist
finitely generated normal subalgebras 
$A \subseteq \Gamma(X, \mathcal{O})^G$, 
which are large in the sense that they 
have $\KK(X)^G$ as their field of fractions,
see Lemma~\ref{lem:invarfield}.
This provides at least candidates 
$\pi' \colon X \to Y'$ with $Y' := \Spec \, A$
for a quotient.
An obvious obstruction to being a categorical 
quotient is non-surjectivity; this even happens
if $X$ is affine and the algebra of invariants
is finitely generated, i.e., we may take 
$A = \Gamma(X, \mathcal{O})^G$.
In general, the image $Y = \pi'(X)$ is a 
constructible set.
This motivates an excursion to the 
category of constructible spaces, i.e., 
spaces locally isomorphic to constructible 
subsets of affine varieties, 
see Section~\ref{sec:crquots} for details 
and~\cite{BB} for a related concept.
We ask whether the map $\pi \colon X \to Y$ 
sending $x \in X$ to $\pi'(x) \in Y$ 
is a {\em categorical quotient in the category
of constructible spaces\/}, i.e., 
every $G$-invariant morphism $X \to Z$ to 
a constructible space $Z$ factors uniquely through 
$\pi \colon X \to Y$;
note that a positive answer allows in particular
to associate a unique quotient to 
the action.
Here comes our first result.

\begin{theorem}
\label{thm:quotchar}
Consider a normal variety $X$ with a factorial 
action of an affine algebraic group $G$.
Let $A \subseteq \Gamma(X, \mathcal{O})^G$
be a finitely generated normal subalgebra having
$\KK(X)^G$ as its field of fractions,
$\pi' \colon X \to Y'$, where $Y' := \Spec \, A$,
the canonical morphism and set $Y := \pi(X)$.
Then the following statements are equivalent.
\begin{enumerate}
\item
The morphism $\pi \colon X \to Y$,
$x \mapsto \pi'(x)$ 
is a categorical quotient in the 
category of constructible spaces 
for the $G$-action on $X$.
\item
The pullback 
$\pi^* \colon \Gamma(Y,\mathcal{O}) \to \Gamma(X,\mathcal{O})^G$
is an isomorphism. 
\item
There is an open subset $Y'' \subseteq Y'$ 
with $Y \subseteq Y''$ and 
$Y'' \setminus Y$ is of codimension at least two in 
$Y''$.
\end{enumerate}
Moreover, if one of these statements holds,
then $\pi \colon X \to Y$ is even 
a strong categorical quotient, 
i.e., $\pi \colon \pi^{-1}(V) \to V$
is a categorical quotient
for every open $V \subseteq Y$.
\end{theorem}

\goodbreak

If the algebra of invariants is finitely generated,
then it obviously fulfills the second condition of 
Theorem~\ref{thm:quotchar},
and thus we obtain the following.

\begin{corollary}
\label{quotexfingen}
Consider a normal variety $X$ with a factorial 
action of an affine algebraic group $G$ and
suppose that $\Gamma(X,\mathcal{O})^G$ is finitely 
generated.
Then $\pi \colon X \to Y$, $x \mapsto \pi'(x)$,
where $\pi' \colon X \to \Spec \, \Gamma(X,\mathcal{O})^G$
is the canonical map and $Y = \pi'(X)$,
is a strong categorical quotient in the category 
of constructible spaces
for the action of $G$ on $X$.
\end{corollary}

We come back to the problem of existence 
of quotients in the category of varieties.
An important observation is that 
$\Gamma(X,\mathcal{O})^G$ admits a 
{\em separating subalgebra\/} in the sense 
of Derksen and Kemper~\cite{DeKe}, i.e., a 
finitely generated subalgebra that separates 
any pair of points, 
which can be separated by invariant functions,
see Proposition~\ref{PDK}.
Combining this with our first result, 
we obtain the following characterization
of existence of categorical quotients.

\begin{theorem}
\label{thm:quotchar2}
Let $X$ be a normal variety with a factorial action 
of an affine algebraic group $G$.
Then the following statements are equivalent.
\begin{enumerate}
\item
There exists a categorical quotient 
$\pi \colon X \to Y$ 
in the category of varieties
for the action of $G$ on $X$.
\item
There is a finitely generated normal subalgebra
$A \subseteq \Gamma(X,\mathcal{O})^G$ 
with quotient field $\KK(X)^G$
such that the canonical map $X \to \Spec \, A$ 
has an open image.
\item
For every normal separating subalgebra 
$A \subseteq \Gamma(X,\mathcal{O})^G$
with quotient field $\KK(X)^G$,
the canonical map $X \to \Spec \, A$ 
has an open image.
\end{enumerate}
Moreover, if one of the statements holds, then 
the categorical quotient $\pi \colon X \to Y$ 
is even a strong categorical quotient.
\end{theorem}

In the case of a finitely generated ring 
of invariants, we obtain as an immediate 
consequence the following 
characterization for existence of a 
categorical quotient.

\begin{corollary}
\label{varquotcharfingen}
Let $X$ be a normal variety with a factorial action 
of an affine algebraic group $G$
and suppose that $\Gamma(X,\mathcal{O})^G$ is finitely 
generated.
Then the following statements are 
equivalent.
\begin{enumerate}
\item
The $G$-action on $X$ has 
a categorical quotient in the 
category of varieties.
\item
The canonical morphism 
$\pi \colon X \to \Spec \, \Gamma(X, \mathcal{O})^G$
has an open image.
\end{enumerate}
Moreover, if one of these statements holds,
then $\pi \colon X \to \pi(X)$ is a categorical
quotient, and it is even a strong one. 
\end{corollary}

For representations of unipotent groups on finite 
dimensional vector spaces, we obtain that existence 
of a categorical quotient in the category of varieties 
is equivalent to being quite close to the reductive 
case.

\begin{theorem}
\label{unipotrepquot}
Let a unipotent group~$G$ act linearly on a finite 
dimensional vector space~$V$.
Then the following statements are equivalent.
\begin{enumerate}
\item
There exists a categorical quotient in the category of
varieties for the action 
of $G$ on $V$.
\item
The algebra $\Gamma(V,\mathcal{O})^G$ of invariants is finitely 
generated and the canonical map 
$V \to \Spec \, \Gamma(V,\mathcal{O})^G$ is surjective.
\end{enumerate} 
Moreover, if one of these statements holds, 
then $V \to \Spec \, \Gamma(V,\mathcal{O})^G$ is
a strong categorical quotient in the category of
varieties for the action of $G$ on $V$.
\end{theorem}

The results presented so far
are proven in Sections~\ref{sec:crquots}
and~\ref{sec:varquots}.
In Section~\ref{sec:exs}, we discuss 
examples.
An application is given in Section~\ref{sec:GIT}.
There, we consider the action of an affine 
algebraic group $G$ with trivial character 
group $\Chi(G)$ on an, e.g. complete variety $X$ 
and assume that the Cox ring $\mathcal{R}(X)$ 
as well as the subring $\mathcal{R}(X)^G$ 
are finitely generated.
Since Cox rings are (graded) factorial~\cite{ArDeHaLa}, 
we can apply our results to the lifted action of $G$
on $\Spec \, \mathcal{R}(X)$.
Via a Gel'fand-MacPherson type correspondence,
we obtain in Construction~\ref{constr:mushroom} 
open $G$-invariant subsets $U \subseteq X$ 
with a strong categorical quotient
from geometric quotients of
a certain torus action on the factorial 
affine variety $\Spec \, \mathcal{R}(X)^G$.
Among the resulting sets $U \subseteq X$, 
there are many sets of 
{\em finitely generated semistable points\/} 
as introduced by Doran and Kirwan in~\cite{DoKi}.
They fit into a combinatorial picture given 
by the GIT-fan of a torus action on 
$\Spec \, \mathcal{R}(X)^G$.

\section{Constructible quotients}
\label{sec:crquots}

In this section, we prove Theorem~\ref{thm:quotchar}.
We begin with presenting the basic concepts concerning 
constructible spaces.

By a space with functions 
we mean a topological 
space $X$ together with a sheaf $\mathcal{O}_X$
of $\KK$-valued functions.
A~morphism of spaces $X$ and $Y$ with functions
is a continuous map $\varphi \colon X \to Y$
such that for every open subset $V \subseteq Y$
and every $g \in \mathcal{O}_Y(V)$, we have 
$g \circ \varphi \in \mathcal{O}_X(\varphi^{-1}(V))$.
If $Y \subseteq X$ is a subset of a space $X$ with 
functions, then $Y$ is in a natural manner 
a subspace with functions: 
firstly, it inherits the subspace topology 
from $X$ and, 
secondly, it inherits the sheaf 
$\mathcal{O}_Y$ of functions that are 
locally represented as restrictions of functions of 
$\mathcal{O}_X$.
A subset $Y \subseteq X$ is called constructible 
if it is a union of finitely many locally 
closed subsets.
By a constructible subspace $Y \subseteq X$, 
we mean a constructible subset $Y \subseteq X$
together with the subspace structure.
We are ready to introduce the 
category of constructible spaces.

\begin{itemize}
\item
A {\em quasiaffine constructible space} 
is a space with functions
isomorphic to a constructible subspace 
of an affine $\KK$-variety. 
\item
A {\em constructible space\/}
is a space with functions
admitting a finite cover by open 
quasiaffine constructible subspaces.
\item
A {\em morphism of constructible spaces\/}
is a morphism of the underlying 
spaces with functions.
\end{itemize}

Note that the prevarieties form a full 
subcategory of the category of 
constructible spaces.
Moreover, every constructible subset 
of a constructible space inherits 
the structure of a constructible space.
We will need the following basic observation.

\begin{lemma}
\label{lem:extfct}
Let $X'$ be a normal affine variety and
$X \subseteq X'$ a dense constructible 
subspace.
If every closed hypersurface $D \subseteq X'$ 
meets $X$, then  the restriction 
$\Gamma(X',\mathcal{O}_{X'}) \to \Gamma(X,\mathcal{O}_X)$
is an isomorphism.
\end{lemma}

\begin{proof}
Locally every $f \in \Gamma(X,\mathcal{O}_X)$
extends to $X'$.
Since $X \subseteq X'$ is dense, the 
local extensions can be glued together 
and thus $f$ extends to an open 
neighbourhood $X'' \subseteq X'$ of $X$.
Normality then gives the claim.
\end{proof}

Similarly one obtains that, given two constructible 
subspaces $X \subseteq X'$ and $Y \subseteq Y'$ 
of varieties $X'$ and $Y'$, every morphism $X \to Y$
extends to a morphism $U' \to Y'$ with an open 
neighbourhood $U' \subseteq X'$ of $X$.
This shows in particular that the category
of dc-subsets defined by A.~Bia\l ynicki-Birula~\cite{BB} 
is a full subcategory of the category of 
constructible spaces.

\begin{proof}[Proof of Theorem~\ref{thm:quotchar}]
In order to obtain ``(i)$\Rightarrow$(ii)'',
apply the universal property of the categorical 
quotient to $G$-invariant functions.

We verify  ``(ii)$\Rightarrow$(iii)''.
Consider $C := Y' \setminus Y$,
let $C_1, \ldots, C_r \subseteq C$ denote
the irreducible components, which are closed 
in $Y'$, and set 
$Y'' := Y' \setminus (C_1 \cup \ldots \cup C_r)$.
By Lemma~\ref{lem:extfct}, we have 
$$ 
\Gamma(Y'', \mathcal{O})
\ = \ 
\Gamma(Y, \mathcal{O}).
$$
We show that $Y'' \setminus Y$ is small.
Otherwise, let $D_1, \ldots, D_s \subseteq Y''$
be the (nonempty) collection of
prime divisors such that 
$D_i \setminus Y$ is dense in $D_i$.
Choose non-zero functions $f,g \in \Gamma(Y',\mathcal{O})$ 
with 
$$ 
D_i \ \subseteq \ V(Y'',f),
\qquad
D_i \ \not\subseteq \ V(Y'',g),
\qquad
V(Y,f) \subseteq \ V(Y,g).
$$
Then, for any $m \in \ZZ_{\ge 0}$,
the function $g^mf^{-1}$ is not regular on $Y''$
and hence not on $Y$.
On the other hand, for $m$ big enough, we have 
$m \, \div(\pi^*(g)) > \div(\pi^*f)$ and thus 
$\pi^*(g^mf^{-1})$ belongs to $\Gamma(X,\mathcal{O})^G$.
This contradicts~(ii).

We check ``(iii)$\Rightarrow$(ii)''.
Clearly, $\pi^* \colon \Gamma(Y,\mathcal{O}) \to \Gamma(X,\mathcal{O})^G$
is injective.
To see surjectivity, let $f \in \Gamma(X,\mathcal{O})^G$
be given.
Then we have $f =\pi^*g  $ with a rational function 
$g \in \KK(Y'')$.
But condition~(iii) ensures that $g$ has no poles
and thus, we have $g \in \Gamma(Y,\mathcal{O})^G$.

We show that (ii) and (iii) imply (i).
Let $\varphi \colon X \to Z$ be a $G$-invariant
morphism. 
Cover $Z$ by open subspaces $Z_1, \ldots, Z_r$ 
such that we have open embeddings $Z_i \subseteq Z_i'$ 
with affine varieties $Z_i'$.
Then $X$ is covered by the open subsets 
$X_i := \varphi^{-1}(Z_i)$, and we have 
\begin{eqnarray*}
X \setminus X_i
& = & 
D_i 
\ \cup \  
B_i,
\end{eqnarray*}
where $D_i \subset X$ is of pure codimension one,
and $B_i \subset X$ is of codimension at least two.
Choose $G$-invariant
functions $f_i \in \Gamma(X,\mathcal{O})$ having
precisely $D_i$ as their set of zeroes.
These $f_i$ descend to $Y$, and,
by Lemma~\ref{lem:extfct},
extend to $Y''$.
Set $Y_i'' := Y''_{f_i}$.
Then we have  
$$
\Gamma(X_i,\mathcal{O})^{G}
\ = \ 
\Gamma(X,\mathcal{O})_{f_i}^{G}
\ = \ 
\pi^* \Gamma(Y,\mathcal{O})_{f_i}
\ = \ 
(\pi')^*\Gamma(Y_i'',\mathcal{O}),
$$
where, for the last equality, we again
use Lemma~\ref{lem:extfct}.
As a consequence, we obtain a 
commutative diagram
$$ 
\xymatrix{
X_i 
\ar[rr]^{\varphi}
\ar[dr]_{\pi'_i}
&&
Z_i'
\\
&
Y_i''
\ar[ur]_{\psi_i'}
&
}
$$
Consider $Y_i := \pi'_i(X_i) \subseteq Y_i''$.
Then we have $Y_i = \pi(X_i)$. 
Moreover, because of 
$\psi_i'(Y_i) = \varphi(X_i) \subseteq Z_i$,
we obtain morphisms 
$\psi_i \colon Y_i \to Z_i$,
$y \mapsto \psi_i'(y)$ 
of constructible spaces.
By construction, these morphisms 
glue together
to the desired factorization 
$\psi \colon Y \to Z$.

In order to see that the categorical 
quotient $\pi \colon X \to Y$ is even strong,
first note that for every principal open subset $Y_f$ the 
restriction $\pi \colon X_{\pi^*f} \to Y_f$ is 
a categorical quotient, because 
it satisfies the second condition of the theorem.
Then the desired property is obtained by gluing.
\end{proof}

\begin{remark}
\label{rem:quaffquot2constr}
Let $G$ act on $X$ as in Theorem~\ref{thm:quotchar}.
If there is a categorical quotient 
$\pi \colon X \to Y$ with a quasiaffine 
constructible space $Y,$ then this 
quotient is obtained by the procedure of
Theorem~\ref{thm:quotchar}.
Indeed, by the universal property of a categorical 
quotient, the pullback
$\pi^* \colon \Gamma(Y,\mathcal{O}) \to \Gamma(X,\mathcal{O})^G$
is an isomorphism. 
Now choose an embedding $Y \subseteq Y'$ 
into an affine variety $Y'$.
Then $A := \pi^*\Gamma(Y',\mathcal{O})$ 
is as wanted.
\end{remark}

Note that, given a subalgebra $A$ of the 
algebra of invariants as in 
Theorem~\ref{thm:quotchar},  
the equivalent conditions of~\ref{thm:quotchar}
need not be fulfilled, see~\cite[Section~4]{Wi2}.

\section{Quotients in the category of varieties}
\label{sec:varquots}

Here, we prove Theorem~\ref{thm:quotchar2}.
A first observation is  existence of 
separating subalgebras; note that 
for affine $G$-varieties, an elementary proof 
is given in~\cite[Theorem~3.15]{DeKe}.

\begin{proposition}
\label{PDK}
Let $G$ be any affine algebraic group and 
$X$ any $G$-variety.
Then there exists a finitely generated separating 
subalgebra $A \subseteq \Gamma(X, \Of)^G$.
Moreover, if $X$ is normal and the $G$-action is
factorial, then one may choose $A$ to be normal and
to have $\KK(X)^G$ as its field of fractions.
\end{proposition}

\begin{lemma}
\label{lem:invarfield}
Let $X$ be a normal variety with a factorial 
action of an affine algebraic group $G$.
Then there is a finitely generated subalgebra 
$A \subseteq \Gamma(X,\mathcal{O})^G$ having $\KK(X)^G$ 
as its field of fractions. 
\end{lemma}

\begin{proof}
Let $\KK(X)^G = \KK(g_1, \ldots, g_r)$ with 
$g_i \in \KK(X)^G$. Then $g_i$ is defined on
an open invariant subset $U_i \subseteq X$.
By factoriality of the action, the union 
of all one-codimensional components of 
$X \setminus U_i$ is the zero set of a
function $f_i \in \Gamma(X,\mathcal{O})^G$.
Normality of $X$ implies  
$h_i := g_if_i^{m_i} \in \Gamma(X,\mathcal{O})^G$
for some $m_i > 0$.
Thus, the algebra 
$A := \KK[f_i,h_i; \; 1 \le i \le r]$ is as wanted.
\end{proof}

\begin{proof}[Proof of Proposition~\ref{PDK}]
Assume that for every finitely generated subalgebra 
$B \subseteq \Gamma(X, \Of)^G$
there exist $x_1, x_2 \in X$ such that 
$F(x_1) = F(x_2)$ for all
$F \in B$, but $f(x_1) \ne f(x_2)$ for some $f \in \Gamma(X, \Of)^G$.
Then we may construct an infinite 
strictly increasing sequence of finitely 
generated subalgebras 
$$
B_1 \ \subset \ B_2 \ \subset \ B_2 \ \subset \ldots
$$
in $\Gamma(X, \Of)^G$ such that for any 
$i \ge 1$ there exist $x_{1i}, x_{2i} \in X$
with $F(x_{1i}) = F(x_{2i})$ for all $F \in B_i$,
but $f(x_{1i}) \ne f(x_{2i})$ for some $f \in B_{i+1}$. 
This sequence of subalgebras gives us the
affine varieties $Y_i := \Spec\, B_i$ 
and the morphisms $\psi_i \colon X \to Y_i$ and 
$\varphi_i \colon Y_{i+1} \to Y_i$
defined by the inclusions $B_i \subset \Gamma(X, \Of)^G$ and
$B_i \subset B_{i+1}$. 

The images $V_i := \psi_i(X) \subseteq Y_i$ and
the maps $\varphi_i \colon V_{i+1} \to V_i$ form a dominated
inverse system of dc-subsets, see~\cite[Section~3]{BB}.
By~\cite[Theorem~O]{BB}, there exists $m \ge 1$ such that the maps
$\varphi_i \colon V_{i+1} \to V_i$ are bijective for any $i \ge m$. 
This implies that the fibers of the morphisms $\psi_i$  and 
$\psi_{i+1}$ coincide for any $i\ge m$, a contradiction.

The supplement is a simple consequence of 
Lemma~\ref{lem:invarfield}
and finite generation of the integral closure. 
\end{proof}

The basic property of a separating 
subalgebra $A \subseteq \Gamma(X,\mathcal{O})^G$ 
we will use is that it realizes the categorical 
closure of the equivalence relation given
by the $G$-action on $X$ in the following
sense.

\begin{proposition}
\label{PP}
Let $X$ be a normal variety with a factorial action 
of an affine algebraic group $G$. 
If $A \subseteq \Gamma(X,\mathcal{O})^G$ is 
a finitely generated  separating subalgebra
and $U \subseteq X$ a $G$-invariant open subset,
then every $G$-invariant morphism 
$\varphi \colon U \to Z$ to a prevariety $Z$
is constant along the fibers of the map
$\pi' \colon X \to \Spec \, A$.
\end{proposition}

\begin{proof}
Consider $x_1, x_2 \in U$ with 
$\varphi(x_1) \ne \varphi(x_2)$. 
Let $Z_1 \subseteq Z$ be an open 
affine neighbourhood of $\varphi(x_1)$. 
Set $U_1 := \varphi^{-1}(Z_1)$, and 
write $U \setminus U_1 = D_1 \cup  B_1$,
where $D_1 \subset U$ is of pure 
codimension one, and $B_1 \subset U$ 
is of codimension at least two.
Then the  morphism $\varphi_1 \colon U_1 \to Z_1$
extends to a morphism 
$\varphi_1 \colon U \setminus D_1 \to Z_1$.
Consequently, we must have
$U_1 = U \setminus D_1$.

Choose a function $f_1 \in \Gamma(X,\mathcal{O})^G$ 
having inside $U$ precisely $D_1$ 
as its set of zeroes.
If $x_2 \in D_1$ holds,
then we obtain $f_1(x_2) = 0$
and $f_1(x_1) \ne 0$. 
If $x_2 \in U_1$ holds,
then there is a function 
$f \in \Gamma(U_1, \Of)^G$ with
$f(x_1) \ne f(x_2)$. 
Since $\Gamma(U_1, \Of)^G$ equals $\Gamma(U, \Of)_{f_1}^G$,
we find a function 
$f' \in \Gamma(U, \Of)^G$ with $f'(x_1) \ne f'(x_2)$.
\end{proof}

\begin{proof}[Proof of Theorem~\ref{thm:quotchar2}]
The implication ``(iii)$\Rightarrow$(ii)'' follows 
from Proposition~\ref{PDK}.
Moreover, ``(ii)$\Rightarrow$(i)''
and the supplement are clear by 
Theorem~\ref{thm:quotchar}.
To verify ``(i)$\Rightarrow$(iii)'',
let $\pi \colon X \to Y$ be a categorical 
quotient.
Given any normal separating subalgebra
$A \subseteq \Gamma(X,\mathcal{O})^G$
as in~(iii),
the universal property yields 
a commutative diagram
$$ 
\xymatrix{
X
\ar[rr]^{\pi'}
\ar[dr]_{\pi}
& &
{\Spec \, A}
\\
&
Y
\ar[ur]_{\psi}
}
$$
By assumption, the morphism 
$\psi \colon Y \to \Spec \, A$ 
is birational.
Moreover, using surjectivity of the 
categorical quotient $\pi \colon X \to Y$ 
and Proposition~\ref{PP}, we see that
it is injective.
Consequently,
since $\Spec \, A$ is normal, Zariski's
Main Theorem yields that 
$\psi \colon Y \to \Spec \, A$ 
is an open embedding.
Using once more surjectivity of
$\pi \colon X \to Y$,
we conclude that $\pi'(X) = \psi(Y)$ 
is open in $\Spec \, A$.
\end{proof}

Every constructible subspace $X \subseteq X'$ 
of a quasiaffine variety has an open kernel,
i.e., a unique maximal subset, which is open 
in the closure of $X$ in $X'$.
This kernel does not depend on the 
embedding $X \subseteq X'$. 
Thus, given an arbitrary constructible space $X$, 
we can define the {\em set of varietic points\/} 
$X^{\rm var} \subseteq X$ as the union of the 
open kernels of its quasiaffine open subspaces.
Note that $X^{\rm var} \subseteq X$ is the 
unique maximal open subspace of $X$, 
which is a prevariety. 
Based on this observation, we obtain
a statement on categorical quotients 
similar to Rosenlicht's theorem on 
existence of an open subset with 
a geometric quotient.

\begin{corollary}
\label{maxopen}
Let $X$ be a normal variety with a factorial action 
of an affine algebraic group $G$.
Then there is a unique maximal invariant 
open subset $U \subseteq X$ 
that admits a categorical quotient 
$\pi \colon U \to V$ in the category 
of varieties.
\end{corollary}

\begin{proof}
Let $A \subseteq \Gamma(X, \mathcal{O})$
be a finitely generated normal separating 
subalgebra.
Then, by Proposition~\ref{PP},
this is a separating subalgebra for any 
invariant open subset of $X$.
Now, set $Y' := \Spec \, A$, let 
$\pi' \colon X \to Y'$ be the canonical
morphism and set $Y := \pi'(X)$.
Then Theorem~\ref{thm:quotchar2} tells 
us that  $U := \pi'^{-1}(V)$ for $V := Y^{\rm var}$
has a categorical quotient in the category of 
varieties.
If another $G$-invariant open set $W \subseteq X$ 
admits a categorical quotient in the category
of varieties,
then Theorem~\ref{thm:quotchar2} yields that 
$\pi'(W)$ is open in $Y'$ and hence 
$\pi'(W) \subseteq V$ holds.
This implies $W \subseteq U$.
\end{proof}

\begin{proof}[Proof of Theorem~\ref{unipotrepquot}]
The supplement and the implication ``(ii)$\Rightarrow$(i)'' are 
direct consequences of Corollary~\ref{varquotcharfingen}.

Now assume that (i) holds. Then Theorem~\ref{thm:quotchar2} provides a normal 
separating subalgebra 
$A \subseteq \Gamma(V,\mathcal{O})^G$ with quotient field $\KK(X)^G$.
Clearly, we may assume that~$A$ is generated by homogeneous polynomials,
i.e.~is a graded subalgebra of $\mathcal{O}(V)$.
Then $\KK^*$ acts on $Y = \Spec \, A$ and $\pi \colon V \to Y$ becomes
equivariant.
The image $\pi(V) \subseteq Y$ is invariant and, according to 
Theorem~\ref{thm:quotchar2}, open in $Y$. Since we have $A_0 = \KK$, we can 
conclude $\pi(V) = Y$.

By~\cite[Sec.~3]{Du},
there are $f_1,\ldots,f_r \in A$  
with $A_{f_i} = \Gamma(V,\mathcal{O})^G_{f_i}$ such that 
the zero set $V_V(f_1,\ldots,f_r)$ is of codimension at least 
two in $V$ and the pullback homomorphism
$$
\pi^* \colon 
\mathcal{O}(Y \setminus V_Y(f_1,\ldots,f_r))
\ \to \ 
\Gamma(V,\mathcal{O})^G
$$ 
is an isomorphism. 
As seen before, $\pi \colon V \to Y$ is surjective.
Consequently, 
the zero set $V_Y(f_1,\ldots,f_r)$ is of codimension at least 
two in $Y$.
Since $Y$ is normal, this implies  $\Gamma(V,\mathcal{O})^G = A$
which finally gives (ii).
\end{proof}

\section{Examples}
\label{sec:exs}

Our first example is an action of the 
additive group $\KK$ on a 
four-dimensional vector space having 
a finitely generated algebra of 
invariants but no categorical quotient 
in the category of varieties.

\begin{example}
\label{nocatquot}
See~\cite[Section~4.3]{PoVi} 
and~\cite[Example~6.4.10]{FSRi}.
We regard $X := \KK^4$ as the 
space of $(2\times 2)$-matrices
and consider the action of the 
additive group $G = \KK$ given by
$$
\lambda \cdot 
\left(
\begin{array}{cc}
a & b \\
c & d
\end{array}
\right)
\ = \
\left(
\begin{array}{cc}
1 & \lambda \\
0 & 1
\end{array}
\right)
\left(
\begin{array}{cc}
a & b \\
c & d
\end{array}
\right)
\ = \
\left(
\begin{array}{cc}
a + \lambda c & b + \lambda d \\
c & d
\end{array}
\right). 
$$
This action fulfills the assumptions 
of Theorem~\ref{thm:quotchar}.
The algebra of invariants 
is generated by $c$, $d$ and $ad-bc$.
The corresponding morphism 
$\pi' \colon \KK^4 \to \KK^3$ has 
the non-open image 
$$
Y \ = V \ \cup \ \{(0,0,0)\},
\qquad
V \ := \ 
\KK^* \times \KK \times \KK \ \cup \  \KK \times \KK^* \times \KK.
$$
According to Corollary~\ref{varquotcharfingen},
there is no categorical quotient in the category 
of varieties.
However, by Corollary~\ref{quotexfingen} the map
$\pi \colon X \to Y$, $x \mapsto \pi'(x)$ 
is a strong categorical quotient in the category of 
constructible spaces.
Moreover the set $U \subseteq X$ of Corollary~\ref{maxopen}
is $\pi^{-1}(V)$.
\end{example}

By a result of Sumihiro, every free torus action 
on a variety admits a geometric quotient with 
a possibly non-separated orbit space.
The following example shows that this is not true for 
actions of the additive group $\KK$, even if 
they admit a categorical quotient in the category 
of constructible spaces.

\begin{example}
See~\cite[Section~5]{Wi1}.
Consider the (non-linear) action of
the additive group $G = \CC$ on 
$X = \CC^4$ defined by
\begin{eqnarray*}
\lambda 
\cdot 
(x_1,\, x_2,\, x_3,\, x_4) 
& := & 
(x_1, \, 
x_2 +  \lambda x_1, \, 
x_3 + \lambda x_2  + \frac{1}{2}\lambda^2x_1, \, 
x_4  + \lambda (x_2^2  -  2x_1x_3-1)).
\end{eqnarray*}
Then this action is free, and, according to~\cite[Lemma~10]{Wi1}
the algebra of invariants is generated by
$$
f_1 \ := \ x_1, 
\qquad 
f_3 \ := \ x_1x_4 -  x_2(x_2^2  -  2x_1x_3-1),
$$
$$
f_2 \ := \ x_2^2  - 2x_1x_3,  
\qquad 
f_4 \, := \, \frac{1}{f_1}(f_3^2 \, - \, f_2(1 \, - \, f_2)^2).
$$
The variety $Y' = \Spec \, \Gamma(\CC^4, \Of)^G
= V(\CC^4; f_1f_4 - f_3^2 + f_2(1  -  f_2)^2)$
is smooth, and the image of the canonical morphism 
$\pi' \colon \CC^4 \to Y'$ is
\begin{eqnarray*}
Y 
& = &
Y'  \ \setminus \
\{f_1 = 0, \, f_2 = 1, \, f_3 = 0, \, f_4  \ne  0\}.
\end{eqnarray*}
Thus, Theorem~\ref{thm:quotchar} says that 
$\pi \colon X \to Y$, $x \mapsto \pi'(x)$ 
is a categorical quotient in the category
of constructible spaces.
Since $\pi$ does not separate the orbits 
of the points $(0,1,0,0)$ and $(0,-1,0,0)$,
a geometric quotient cannot exist, even if we allow 
a non-separated orbit space.
\end{example}

So far, we saw examples of unipotent group actions 
having no categorical quotients in the category 
of varieties.
Here comes a semisimple group action on a smooth 
quasiaffine variety.

\begin{example}
Let $V$ be the space of $(2\times 3)$-matrices
with the $\SL(2)$-action by left multiplication. The algebra
of invariants is generated by $(2\times 2)$-minors
$\Delta_{12}, \Delta_{23}, \Delta_{13}$, and the canonical morphism
$$
\pi' \colon  V \ \to \ \KK^3, \qquad 
 \ M \to (\Delta_{12}(M), \, \Delta_{23}(M), \, \Delta_{13}(M))
$$
is surjective. Consider the open
invariant subset $X \subset V$ consisting of matrices
with non-zero first column. 
It has the same algebra of invariants as $V$.
However, by Corollary~\ref{varquotcharfingen}, it has 
no categorical quotient, because the image 
$Y = \pi'(X) \subset \KK^3$ is not open: it is 
given by
$$
(\KK^3 \setminus V( \KK^3 ; \, \Delta_{12}, \, \Delta_{13})) 
\ \cup \ 
\{(0,0,0)\}.
$$  
\end{example}

We now provide a class of examples, showing that the conditions 
of Theorem~\ref{thm:quotchar2} may be fulfilled even without
finite generation of the ring of invariants.

\begin{example}
Let $F$ be a connected simply connected semisimple 
algebraic group and $G \subseteq F$ a closed subgroup
with $\Chi(G) = 0$, and let $G$ act on $F$ by 
multiplication from the right.
Then, in general, $\Gamma(F, \Of)^G$
is not finitely generated.
Choose any finitely generated normal subalgebra 
$A \subseteq \Gamma(F, \Of)^G$ having $\KK(F)^G$ 
as its field of fractions and being 
invariant with respect to the $F$-action by 
multiplication from the left.
Then the morphism 
$\pi' \colon F \to Y' := \Spec \, A$ is 
$F$-equivariant and its image coincides 
with an open $F$-orbit.  
\end{example}

The next example shows that without 
the assumption of a ``factorial action'',
even a surjective morphism
$\pi' \colon X \to \Spec \, \Gamma(X, \Of)^G$ 
need not be a categorical quotient.

\begin{example}
Consider the action of the additive group 
$G=\KK$ on the smooth quasiaffine variety
\begin{eqnarray*}
X 
& = &
V(\KK^4; \; x_1x_4 - x_2x_3) \ \setminus \,\{(0,0,0,0)\}  
\end{eqnarray*}
given by
\begin{eqnarray*}
\lambda \cdot  (x_1, x_2, x_3, x_4) 
& := &
(x_1, x_2, x_3 + \lambda x_1, x_4  + \lambda x_2).
\end{eqnarray*}
The algebra of invariants is generated 
by $x_1$ and $x_2$, and the canonical morphism
$\pi' \colon X \to  \Spec \, \Gamma(X, \Of)^G$ 
is surjective. However, the following 
$G$-invariant morphism does not factor 
through $\pi'$:
$$
X \ \to \ \PP_1, 
\qquad \qquad 
x \ \mapsto \ [x_1, x_2] \ = \ [x_3, x_4].
$$
\end{example}

Finally, we give an example without quotient in the 
category of varieties, where we don't know, if it 
has a quotient in the category of constructible 
spaces:
  
\begin{example}
Fix a number $m \in \ZZ_{\ge 2}$ and 
consider the action of the additive 
group $G = \CC$ on  $X = \CC^7$ given by
\begin{eqnarray*}
\lambda \cdot (x, y, z, s, t, u, v) 
& := &  
(x, y,z, s + \lambda x^{m+1},  
t + \lambda y^{m+1}, 
u + \lambda z^{m+1}, 
v + \lambda x^my^mz^m).
\end{eqnarray*}
As observed in~\cite{AC}, the algebra of invariants is
Roberts' algebra~\cite{Ro}; in particular, 
it is not finitely generated. 
By~\cite[Lemma~2]{Ro}, any non-constant term of a 
$G$-invariant polynomial 
contains at least one of the variables $x, y$ and $z$. 
Let
$$
f_1 = x, \ 
f_2 = y, \
f_3 = z, \ 
f_4, \  \ldots, \ f_n
\ \in \ 
\Gamma(X,\mathcal{O})^G
$$
generate a normal separating subalgebra 
and suppose that none of the $f_i$ 
has a constant term. 
Consider the morphism $\pi' \colon \CC^7 \to \CC^n$
given by 
\begin{eqnarray*}
(x, y, z, s, t, u, v) 
& \mapsto & 
(x, y, z, f_4(x,y,z,s,t,u,v), \ldots, f_n(x,y,z,s,t,u,v)).
\end{eqnarray*}

\goodbreak 

We claim that the image $Y = \pi'(\CC^7)$ is not 
open in its closure.
Otherwise, it were a 6-dimensional variety. 
But if we restrict the projection 
$$
r  \colon  \CC^n \ \to \CC^3, 
\qquad (x, y, z, \dots) \ \mapsto \ (x, y, z)
$$
to $Y$, then the preimage $r^{-1}(0,0,0)$ intersected with
$Y$ is just one point; a constradiction to 
semicontinuity of the fiber dimension. 
Thus, by Theorem~\ref{thm:quotchar2},
there is no categorical quotient in the category of varieties.
\end{example}

\section{A combinatorial GIT-type construction}
\label{sec:GIT}

Let $G$ be an affine algebraic group with 
trivial character group $\Chi(G)$.
We consider an action of $G$ on a $\QQ$-factorial
variety $X$ with  $\Gamma(X,\mathcal{O}^*) = \KK^*$
and finitely generated divisor class group
$\Cl(X)$.
For simplicity, let us assume that $\Cl(X)$ 
is free, though m.m. everything works as well
if torsion appears.
Our aim is to present a construction of open
$G$-invariant subsets $U \subseteq X$ that admit 
a strong categorical quotient $U \to Y$.
Passing, if necessary, to the action of 
the simply connected covering group, we may assume 
that $G$ itself is simply connected.

The idea is to lift the $G$-action to the
characteristic space over $X$ and then reduce 
the problem to 
the case of a torus action on an affine variety 
by means of the results obtained so far. 
More precisely, the procedure is the following;
we refer to~\cite{ArDeHaLa} for details.
Choose any subgroup $K \subseteq \WDiv(X)$ of the 
group of Weil divisors projecting isomorphically 
onto the divisor class group $\Cl(X)$ 
and define a sheaf of 
$K$-graded $\mathcal{O}_X$-algebras by
\begin{eqnarray*}
\mathcal{R}
& := & 
\bigoplus_{D \in K} \mathcal{O}_X(D).
\end{eqnarray*}
Then the $K$-grading of $\mathcal{R}$ 
defines an action of the torus 
$H := \Spec \, \KK[K]$
on the relative spectrum 
$\rq{X} := \Spec_X \mathcal{R}$ 
and the canonical morphism 
$p \colon \rq{X} \to X$ is a geometric quotient for 
this action.
We call $p \colon \rq{X} \to X$ the characteristic space
over $X$; for smooth $X$, we obtain the well-known 
universal torsor.
Using $G$-linearization of the homogeneous components
of $\mathcal{R}$, we may lift the 
$G$-action to $\rq{X}$ such that it commutes 
with the $H$-action and $p \colon \rq{X} \to X$
becomes $G$-equivariant, see~\cite[Section~4]{ArHa}.

The variety $\rq{X}$ is quasiaffine and 
the Cox ring 
$\mathcal{R}(X) = \Gamma(\rq{X},\mathcal{O})$ 
is factorial.
In particular, the $G$-action on $\rq{X}$ 
satisfies the assumptions of Theorems~\ref{thm:quotchar}
and~\ref{thm:quotchar2}.
Suppose that the Cox ring 
$\mathcal{R}(X)$ 
and the algebra of invariants 
$\mathcal{R}(X)^G$ 
are finitely generated.
This gives us factorial affine varieties
$$
\b{X} \ := \ \Spec \, \mathcal{R}(X),
\qquad\qquad
\b{Y}' \ := \ \Spec \, \mathcal{R}(X)^G,
$$
see~\cite[Theorem~3.17]{PoVi}.
The variety $\rq{X}$ is an $(G \times H)$-invariant open 
subset of $\b{X}$ and,
by Corollary~\ref{quotexfingen},
there is a strong categorical quotient
$\kappa \colon \b{X} \to \b{Y}$ 
with a constructible subset 
$\b{Y} \subseteq \b{Y}'$ such that 
$\b{Y}' \setminus \b{Y}$ is of codimension at
least two.
Moreover, since 
$\mathcal{R}(X)^G$ is $K$-graded,
the $H$-action on $\b{X}$ 
descends to an $H$-action on $\b{Y}'$
leaving $\b{Y}$ invariant.

\begin{construction}
\label{constr:mushroom}
Let $\rq{V}' \subseteq \b{Y}'$ be an 
$H$-invariant open subset with
$\kappa^{-1}(\rq{V}') \subseteq \rq{X}$
admitting a good quotient 
$q' \colon \rq{V}' \to V'$ 
for the action of $H$.
Set $\rq{V} :=  \b{Y} \cap \rq{V}'$
and suppose we have ($*$): for each
$v \in V := q(\rq{V})$,
the closed $H$-orbit of $q'^{-1}(v)$ 
lies in~$\rq{V}$.
Then $U := p(\rq{U})$,
where $\rq{U} := \kappa^{-1}(\rq{V})$,
is open in $X$,
admits a strong categorical quotient 
$r \colon U \to V$  for the action of $G$
in the category of constructible spaces
and $U$ is covered by $r$-saturated 
affine open subsets.
For convenience, we summarize the data in 
a commutative diagram:
$$ 
\xymatrix{
{\b{X}}
\ar@/^3pc/[rrrrrr]^{\kappa}
&
{\rq{X}}
\ar[d]_{p}
\ar@{}[l]|{\supseteq}
& 
{\rq{U}}
\ar[d]_{p}
\ar@{}[l]|{\supseteq}
\ar[rr]^{\kappa}
& & 
{\rq{V}}
\ar@{}[r]|{\subseteq}
\ar[dd]^{q}
&
{\rq{V}'}
\ar@{}[r]|{\subseteq}
\ar[dd]^{q'}
&
{\b{Y}'}
\\
&
X
& 
U
\ar@{}[l]|{\supseteq}
\ar[d]_{r}
& & 
&
\\
&
& 
V
\ar@{=}[rr]_{\cong}
& &
V 
\ar@{}[r]|{\subseteq}
&
V'
}
$$
\end{construction}

\begin{lemma}
\label{lem:good2constr}
Let a reductive group $H$ act
on a normal variety $\rq{V}'$ 
with  good quotient 
$q' \colon \rq{V}' \to V'$
and let 
$\rq{V} \subseteq \rq{V}'$ be an $H$-invariant 
constructible subset.
If $\rq{V}' \setminus \rq{V}$ is of codimension 
at least two in $\rq{V}'$ and for every 
$v \in V := q'(\rq{V})$ the closed $H$-orbit
of $q'^{-1}(v)$ lies in $\rq{V}$,
then $q \colon \rq{V} \to V$, $x \mapsto q'(x)$ 
is a strong categorical quotient 
for the action of $H$ on $\rq{V}$
in the category of constructible spaces.
\end{lemma}

\begin{proof}
Let $\varphi \colon \rq{V} \to Z$ be any 
$H$-invariant morphism to a constructible 
space.
By assumption, we have $\varphi = \psi \circ q$ 
with a set-theoretical map $\psi \colon V \to Z$.
In order to see that this map is a morphism
note that firstly $V$ carries the 
quotient topology with respect to
$q \colon \rq{V} \to V$, because $V'$ 
carries the quotient topology 
with respect to $q' \colon \rq{V}' \to V'$,
and secondly, that due to the fact that 
$\rq{V}' \setminus \rq{V}$ is of codimension 
at least two, the canonical morphism 
$\mathcal{O}_V \to q_* \mathcal{O}_{\rq{V}}^H$
is an isomorphism.
Clearly, the arguments work as well locally
with respect to $V$, and thus 
we have even a strong categorical quotient.
\end{proof}

\begin{proof}[Proof of Construction~\ref{constr:mushroom}]
Since $q \colon \rq{V}' \to V'$ is a 
good quotient, the set $\rq{V}'$
is covered by $q$-saturated affine open 
subsets, and these are of the form 
$\b{Y}'_{g_i}$.
Thus, $\rq{U} := \kappa^{-1}(\rq{V}')$ 
is covered by the $q \circ \kappa$-saturated 
open subsets $\b{X}_{f_i}$, where 
$f_i := \kappa^*(g_i)$.
Since $\rq{U}$ is $(G \times H)$-invariant, 
its image $U = p(\rq{U})$ is open 
and $G$-invariant. Moreover,
$U$ is covered by the $G$-invariant 
affine open subsets $U_i := p(\b{X}_{f_i})$.
Since $p \colon \rq{U} \to U$ is a 
categorical quotient, 
we have an induced morphism 
$r \colon U \to V$.
By Lemma~\ref{lem:good2constr}
and Theorem~\ref{thm:quotchar},
this is a strong categorical
quotient.
Moreover, by construction,
the sets $U_i$ give the desired
$r$-saturated affine covering.
\end{proof}

Now, in addition to the assumptions 
made so far, let $X$ be projective.
For every Weil divisor $D \in K$,
we may define the associated set of 
semistable points
$X^{ss}(D)$ as the union of all 
the affine sets $X_f$, where 
$n > 0$ and $f \in \mathcal{R}(X)^G_{nD}$.
Then, for any ample divisor $D \in K$,
we have 
$$
X^{ss}(D)
\ = \
p(\kappa^{-1}(\b{Y}'^{ss}(D))),
\qquad\qquad
\b{Y}'^{ss}(D)
\ := \ 
\bigcup_{{f \in \mathcal{R}(X)^G_{nD}} \atop {n > 0}} \b{Y}'_f.
$$
Note that due to our finiteness assumptions 
on the Cox ring $\mathcal{R}(X)$ 
and the ring $\mathcal{R}(X)^G$ of invariants,
the set $X^{ss}(D)$ coincides with the set 
of finitely generated semistable points 
introduced in~\cite[Definition~4.2.6]{DoKi}.
Applying Construction~\ref{constr:mushroom}
shows existence of a categorial quotient.

\begin{corollary}
Let $D \in K$ and suppose that 
$\rq{V}' = \b{Y}'^{ss}(D)$ satisfies 
Condition~($*$) of~\ref{constr:mushroom},
e.g., all points of $\rq{V}'$ 
are stable.
Then  there is a strong categorical 
quotient $X^{ss}(D) \to V$ for the $G$-action,
where $V = q(\b{Y} \cap \rq{V}')$
and $q \colon  \rq{V}' \to \rq{V}' \quot H$
is the good quotient.
\end{corollary}

Now one may apply the combinatorial description of 
GIT-equivalence for torus actions on factorial affine 
varieties, see~\cite[Section~3]{ArHa}, 
to the action of $H$ on $\b{Y}'$,
and thus compute the variation of the 
Doran-Kirwan GIT-quotients. 
We demonstrate this by means of the 
following example.

\begin{example}
Compare also~\cite[Example~4.1.10]{DoKi}.
Consider the action of the additive group 
$G = \KK$ on $X = \PP_1 \times \PP_1$
given by 
\begin{eqnarray*}
\lambda \cdot ([a,c], [b,d])
& := &
([a+\lambda c,c], [b + \lambda d, d]). 
\end{eqnarray*} 
We have an obvious lifting of the action 
to the characteristic space $\rq{X}$:
The extension to $\b{X} = \KK^2 \times \KK^2$
was discussed in Example~\ref{nocatquot},
we have $\b{Y}' = \KK^3$ and the quotient map is
$$ 
\pi' \colon \b{X} \ \to \ \b{Y}',
\qquad
((a,c), (b,d)) \ \mapsto \ (c,d,ad-bc).
$$
The image is 
$\b{Y} = 
\KK^* \times \KK \times \KK
\cup 
\KK \times \KK^* \times \KK
\cup
\{(0,0,0)\}
$.
Now, the torus $H$ is $\KK^* \times \KK^*$ and
it acts on $\b{X}$ via 
\begin{eqnarray*}
(t_1,t_2) \cdot ((a,c),(b,d))
& = & 
((t_1a,t_1c),(t_2b,t_2d)).
\end{eqnarray*}
The induced $H$-action on $\b{Y}'$ 
is given by 
$t \mal (u,v,w) = (t_1u,t_2v,t_1t_2w)$.
Its GIT-fan in $\Chi(H) = \ZZ^2$ looks like
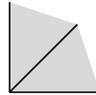
\begin{center}
\input{p1p1hirz.pstex_t}
\end{center}
The two full-dimensional chambers correspond via 
Construction~\ref{constr:mushroom} to the two
sets $U_1 := \PP_1 \times \KK$ and $U_2 := \KK \times \PP_1$ 
of semistable points.
Both of them have a strong categorical quotient 
$U_i \to \PP_1$ in the category of varieties.   
\end{example}

\end{document}

%% file: p1p1hirz.pstex_t
\begin{picture}(0,0)%
\includegraphics{p1p1hirz.pstex}%
\end{picture}%
\setlength{\unitlength}{1243sp}%
\begingroup\makeatletter\ifx\SetFigFontNFSS\undefined%
\gdef\SetFigFontNFSS#1#2#3#4#5{%
  \reset@font\fontsize{#1}{#2pt}%
  \fontfamily{#3}\fontseries{#4}\fontshape{#5}%
  \selectfont}%
\fi\endgroup%
\begin{picture}(1866,1866)(2218,-1444)
\end{picture}%

%% file: unipot.bbl
\begin{thebibliography}{}%
%
\bibitem{AC}
A.~A'Campo-Neuen:
Note on a counterexample to Hilbert's fourteenth problem given by P.~Roberts.  
Indag. Math. (N.S.) 5 (1994), no.~3, 253--257.
%
\bibitem{ACHa}
A.~A'Campo-Neuen, J.~Hausen: 
Examples and counterexamples for existence of categorical quotients.  
J. Algebra 231 (2000), no.~1, 67--85.
%
\bibitem{ArDeHaLa} 
I.~Arzhantsev, U.~Derenthal, J.~Hausen, A.~Laface:
Cox rings, arXiv:1003.4229, see also the authors' 
webpages.
%
%
\bibitem{ArHa}
I.V.~Arzhantsev, J.~Hausen: 
Geometric Invariant Theory via Cox rings.  
J.~Pure Appl. Algebra 213 (2009), no.~1, 154--172.
%
%
\bibitem{BB}
A.~Bialynicki-Birula:
Categorical quotients.  
J. Algebra  239 (2001), no.~1, 35--55.
%
%
\bibitem{DeKe}
H.~Derksen, G.~Kemper:
Computational Invariant Theory.
Invariant Theory and Algebraic Transformation Groups, I, 
Encyclopaedia of Math. Sciences, 130, 
Springer-Verlag, Berlin, 2002 
%
\bibitem{DoKi}
B.~Doran, F.~Kirwan:
Towards non-reductive geometric invariant theory.  
Pure Appl. Math. Q. 3 (2007), no.~1, part~3, 61--105.
%
\bibitem{Du}
E.~Dufresne: Finite separating sets and quasi-affine quotients.
arXiv:1102.2132v1 (2011).
%
\bibitem{Fa}
A.~Fauntleroy:
Categorical quotients of certain algebraic group actions. 
Illinois J. Math. 27 (1983), no.~1, 115--124. 
%
\bibitem{FSRi}
W.~Ferrer Santos, A.~Rittatore:
Actions and invariants of algebraic groups. 
Pure and Applied Mathematics (Boca Raton), 269. 
Chapman and Hall/CRC, Boca Raton, FL, 2005
%
\bibitem{PoVi}
V.L.~Popov, E.B.~Vinberg: 
Invariant Theory. 
In Algebraic Geometry IV, 137--314, Moscow: VINITI 1989 (Russian); 
English transl.: Algebraic Geometry IV, Encyclopaedia Math. Sci., 
vol.~55, 123--284, Springer-Verlag, Berlin,  1994
%
\bibitem{ReRi}
L.~Renner, A.~Rittatore:
Observable actions of algebraic groups.  
Transform. Groups  14  (2009),  no. 4, 
985--999. 
%
\bibitem{Ro}
P.~Roberts:
An infinitely generated symbolic blow-up 
in a power series ring and a new counterexample 
to Hilbert's fourteenth problem. 
J.~Algebra 132 (1990), no.~2, 461--473. 
%
\bibitem{Wi1}
J.~Winkelmann: 
On free holomorphic $\bf C$-actions on $C\sp n$ 
and homogeneous Stein manifolds.  
Math. Ann. 286 (1990), no.~1-3, 593--612.
%
\bibitem{Wi2}
J.~Winkelmann:
Invariant rings and quasiaffine quotients.  
Math. Z. 244 (2003), no.~1, 163--174.
%
\end{thebibliography}
